\documentclass[12pt]{amsart}

\usepackage[margin=2.4cm]{geometry}
\usepackage{a4wide}

\usepackage{amsmath,amssymb,amsthm}

\newtheorem{thm}{Theorem}[section]
\newtheorem{lem}[thm]{Lemma}

\theoremstyle{remark}

\numberwithin{equation}{section}

\begin{document}

\title[On a Zolt\'{a}n Boros' problem]{On a Zolt\'{a}n Boros' problem connected with polynomial-like iterative equations}
\author[S. Draga, J. Morawiec]{Szymon Draga, Janusz Morawiec}
\email[S. Draga]{szymon.draga@gmail.com}
\email[J. Morawiec]{morawiec@math.us.edu.pl}
\date{}
\address{Institute of Mathematics, University of Silesia, Bankowa 14, 40-007 Katowice, Poland}

\begin{abstract}
We determine all continuous solutions $g\colon I\to I$ of the  polynomial-like iterative equation $g^3(x)=3g(x)-2x$, where $I\subset\mathbb R$ is an interval. In particular, we obtain an answer to a problem posed by Zolt\'{a}n Boros (during the Fiftieth International Symposium on Functional Equations, 2012) of determining all continuous functions $f\colon (0,+\infty)\to (0,+\infty)$ satisfying $f^3(x)=\frac{[f(x)]^3}{x^2}$.
\end{abstract}


\maketitle

\section{Introduction}
Given an interval $I\subset\mathbb R$ we are interested in determining all continuous functions ${g\colon I\to I}$ satisfying
\begin{equation}\label{E}
g^3(x)=3g(x)-2x.
\end{equation}
Here and throughout the paper $g^n$ denotes the $n$-th iterate of a given self-mapping $g\colon I\to I$; i.e., $g^0={\rm id}_I$ and $g^k=g\circ g^{k-1}$ for all integers $k\geq 1$.

There are two reasons to find all continuous solutions  $g\colon I\to I$ of equation \eqref{E}.
The first one is to answer a problem posed by Zolt\'{a}n Boros (see \cite{B}) of determining all continuous functions $f\colon (0,+\infty)\to (0,+\infty)$ satisfying 
\begin{equation}\label{be}
f^3(x)=\frac{[f(x)]^3}{x^2}.
\end{equation}

The second reason is that equation \eqref{E} belongs to the class of important and intensively investigated iterative functional equations; i.e., the class of polynomial-like iterative equations of the form
\begin{equation}\label{ple}
\sum_{n=0}^Na_ng^n(x)=F(x),
\end{equation}
where $a_n$'s are given real numbers, $F\colon I\to I$ is a given function and $g\colon I\to I$ is the unknown function.  For the theory of equation \eqref{ple} and its generalizations we refer the readers to books \cite{KCG1990,T1981}, surveys \cite{BJ2001,ZYZ1995}, and some recent papers \cite{CZ2009, G2014, GZ2013, LZ2013, MZ2000, T2010, TZ2004, X2004, XZ2007a, XZ2007b, ZNX2006, ZXZ2013}. 

Equation \eqref{ple} represents a linear dependence of iterates of the unknown function and looks like a linear ordinary differential equation with constant coefficients, expressing the linear dependence of derivatives of  the unknown function. The difference between \pagebreak these two equations is that  linear ordinary differential equations with constant coefficients have a complete theory for finding their solutions, in contrast, even to a very interesting subclass of homogeneous polynomial-like iterative equations of the form
\begin{equation}\label{sple}
\sum_{n=0}^Na_ng^n(x)=0.
\end{equation}
The difficulties in solving equation \eqref{sple}, and hence also equation \eqref{ple}, comes from the fact that the iteration operator $g\to g^n$ is non-linear. 

The problem of finding all continuous solutions of equation \eqref{sple} for a given positive integer $N$ seems to be very difficult.  It is completely solved in \cite{N1974} (see also \cite{MZ1998}) for $N=2$, but  it is still  open even in the case where $N=3$ (see \cite{M1989}). It turns out that the nature of continuous solutions of equation \eqref{ple} depends deeply on the behavior of complex roots of its characteristic equation 
\begin{equation}\label{che}
\sum_{n=0}^N\alpha_nr^n=0.
\end{equation}
This characteristic equation is motivated by the Euler's idea for differential equations; it is obtained by putting $g(x)=rx$ into \eqref{sple} to determine all its linear solutions. There are some results describing all continuous solutions of equation \eqref{sple} with $N\geq 3$ in very particular cases where the complex roots of equation \eqref{che} fulfill special conditions (see \cite{YZ2004, ZG, ZZ2010}).  Note that the characteristic equation of equation \eqref{E} is of the form
$$
r^3-3r+2=0,
$$ 
and it has two roots: $r_1=1$ of multiplicity~$2$ and $r_2=-2$ of multiplicity~$1$. Therefore, none of known results can be used to determine all continuous solutions $g\colon I\to I$ of equation \eqref{E}.

\section{Preliminary}
It is easy to check that the identity function, defined on an arbitrary set $A\subset\mathbb R$, is a continuous solution of equation \eqref{E}. Thus equation \eqref{E} has the unique solution $g\colon I\to I$ in the case where $I\subset\mathbb R$ is an interval degenerated to a single point. Therefore, from now on, fix a non-degenerated interval $I\subset\mathbb R$;  open or closed or closed on one side, possible infinite.

\begin{lem}\label{lem11}
Assume that $g\colon I\to I$  is a continuous solution of equation \eqref{E}. Then $g$ is strictly monotone.
Moreover, if $I\neq\mathbb R$, then $g$ is strictly increasing.
\end{lem}

\begin{proof}
Fix $x,y\in I$ and assume that $g(x)=g(y)$. Then by \eqref{E} we obtain
$$
x=\frac{3g(x)-g^3(x)}{2}=\frac{3g(y)-g^3(y)}{2}=y.
$$ 
Since $g$ is continuous, it follows that it is strictly monotone.

Assume now that $I\neq\mathbb R$ and suppose that, contrary to our claim, $g$ is strictly decreasing.
Put $a=\inf I$ and $b=\sup I$. If $b=+\infty$, then 
$$
-\infty<a\leq\lim_{x\to b}g^3(x)=\lim_{x\to b}\big(3g(x)-2x\big)=-\infty,
$$ 
a contradiction. Similarly, if $a=-\infty$, then 
$$
+\infty>b\geq\lim_{x\to a}g^3(x)=\lim_{x\to a}\big(3g(x)-2x\big)=+\infty,
$$
a contradiction. Therefore, we have proved that $a,b\in\mathbb R$. Put $c=\inf g(I)$. Since $I$ is non-degenerated and $g$ is strictly decreasing, we have $a\leq c<b$. Moreover, by \eqref{E} we have
$$
c\leq\lim_{x\to b}g^3(x)=\lim_{x\to b}\big(3g(x)-2x\big)=3c-2b,
$$
a contradiction.
\end{proof}

\begin{lem}\label{lem12}
Every continuous solution $g\colon\mathbb R\to\mathbb R$ of equation \eqref{E} maps bijectively $\mathbb R$ onto $\mathbb R$.
\end{lem}

\begin{proof}
According to Lemma \ref{lem11} it is enough to show that $\lim_{x\to +\infty}g(x)\in\{-\infty,+\infty\}$ and  $\lim_{x\to -\infty}g(x)\in\{-\infty,+\infty\}$.

Suppose, towards a contradiction, that $\lim_{x\to +\infty}g(x)\in\mathbb R$. By the continuity of $g$ we have
$$
g^2\big(\lim_{x\to +\infty}g(x)\big)=\lim_{x\to +\infty}g^3(x)=\lim_{x\to +\infty}\big(3g(x)-2x\big)=-\infty,
$$
a contradiction. In the same manner we obtain $\lim_{x\to -\infty}g(x)\in\{-\infty,+\infty\}$.
\end{proof}

\begin{lem}\label{lem21}
Define sequences $(a_n)_{n\in\mathbb N_0}$, $(b_n)_{n\in\mathbb N_0}$ and $(c_n)_{n\in\mathbb N_0}$ by putting 
$$
a_0=0,\hspace{3ex}b_0=3,\hspace{3ex}c_0=-2
$$
and
$$
a_{n+1}=b_n,\hspace{3ex}b_{n+1}=3a_n+c_n,\hspace{3ex}c_{n+1}=-2a_n
\hspace{3ex}\hbox{for every }n\in\mathbb N_0.
$$
Furthermore, assume that $g\colon I\to I$ solves \eqref{E}. Then
\begin{equation}\label{en}
g^{n+3}(x)=a_ng^2(x)+b_ng(x)+c_nx
\end{equation}
for all $n\in\mathbb N_0$ and $x\in I$.

Moreover, for every $n\in\mathbb N_0$, the following assertions hold:
\begin{itemize}
\item[\rm (i)] $a_n+b_n+c_n=1$;
\item[\rm (ii)] $b_{n+1}-b_n=\sum_{k=0}^{n+3}(-2)^k$;
\item[\rm (iii)] $b_n=\frac{1}{9}[(-2)^{n+4}+3n+11]$.
\end{itemize}
\end{lem}

\begin{proof} 
The proof is by induction on $n\in\mathbb N_0$.

To prove the main part of the lemma if is enough to observe that putting $g(x)$ instead of $x$ in \eqref{en} and making use of  \eqref{E} we obtain 
$$
g^{n+4}(x)=a_ng^3(x)+b_ng^2(x)+c_ng(x)=b_ng^2(x)+(3a_n+c_n)g(x)-2a_nx
$$ 
for every $x\in I$.

(i)  Since $a_0+b_0+c_0=1$ and $a_{n+1}+b_{n+1}+c_{n+1}=a_n+b_n+c_n$, the assertion follows.

(ii) From assertion (i)  we have 
$$
b_{n+1}-b_n=1-a_{n+1}-c_{n+1}-b_n=-2[b_n-b_{n-1}]+1.
$$
Now we need only to observe that $b_1-b_0=-5=\sum_{k=0}^{3}(-2)^k$.

(iii) Clearly, $b_0=3=\frac{1}{9}[(-2)^4+11]$. Moreover, by assertion (ii) we have
$$
b_{n+1}=(b_{n+1}-b_n)+b_n=\frac{1-(-2)^{n+4}}{3}+b_n=\frac{1}{9}[(-2)^{n+5}+3(n+1)+11],
$$
which completes the proof.
\end{proof}

From now on $(a_n)_{n\in\mathbb N_0}$, $(b_n)_{n\in\mathbb N_0}$ and $(c_n)_{n\in\mathbb N_0}$ will stand for the sequences defined in the foregoing lemma.

\begin{lem}\label{lem22}
Assume that $g\colon I\to I$ is a continuous solution of equation \eqref{E}. Then 
\begin{equation}\label{22}
\lim_{n\to\infty}\frac{g^{n+3}(x)}{b_n}=-\frac{1}{2}g^2(x)+g(x)-\frac{1}{2}x \hspace{3ex}\hbox{ for every }x\in I.
\end{equation} 
\end{lem}

\begin{proof} 
From Lemma \ref{lem21} we conclude that $\lim_{n\to\infty}\frac{1}{b_n}=0$ and $\lim_{n\to\infty}\frac{a_n}{b_n}=\lim_{n\to\infty}\frac{c_n}{b_n}=-\frac{1}{2}$. Dividing both sides of \eqref{en} by $b_n$ and next tending with $n$ to infinity we obtain \eqref{22}.
\end{proof}

\begin{lem}\label{lem23}
Assume that $g\colon I\to I$ is a continuous solution of equation \eqref{E}. If for every $x\in I$ the sequence $(g^n(x))_{n\in\mathbb N}$ converges to a real number, then 
\begin{equation}\label{23}
g(x)=x \hspace{3ex}\hbox{ for every }x\in I.
\end{equation}
\end{lem}

\begin{proof} 
As in the previous proof we obtain $\lim_{n\to\infty}\frac{1}{b_n}=0$. Then Lemma \ref{lem22} implies
\begin{equation}\label{24}
g^2(x)-g(x)=g(x)-x \hspace{3ex}\hbox{ for every }x\in I.
\end{equation}  
By a simple induction we obtain
\begin{equation}\label{25}
g^{n+1}(x)-g^n(x)=g(x)-x \hspace{3ex}\hbox{ for all }n\in\mathbb N\hbox{ and }x\in I.
\end{equation}  
Finally, tending with $n$ to infinity in \eqref{25} we come to \eqref{23}.
\end{proof}

\begin{lem}\label{lem24}
Assume that $g\colon I\to I$ is a continuous solution of equation \eqref{E}.
\begin{itemize}
\item[\rm (i)] If for some $x\in I$ we have $g^2(x)-2g(x)+x\neq 0$, then
\begin{equation}\label{26}
\lim_{n\to\infty}\frac{g^{n+4}(x)}{g^{n+3}(x)}=-2.
\end{equation}
\item[\rm (ii)] If $g$ is increasing, then \eqref{25} holds.
\end{itemize}
\end{lem}

\begin{proof} 
(i)  Lemmas \ref{lem22} and \ref{lem21} yield
$$
\lim_{n\to\infty}\frac{g^{n+4}(x)}{g^{n+3}(x)}=\lim_{n\to\infty}\frac{g^{n+4}(x)}{b_{n+1}} \frac{b_n}{g^{n+3}(x)}\frac{b_{n+1}}{b_n}=\lim_{n\to\infty}\frac{b_{n+1}}{b_n}=-2.
$$

(ii) If $g$ is increasing, then the sequence $(g^n(x))_{n\in\mathbb N}$ is monotone for every $x\in I$. Hence $\lim_{n\to\infty}g^n(x)$ exists and it equals either a real number or $\pm\infty$. In both the cases \eqref{26} cannot be satisfied. Then by assertion (i) we see that \eqref{24} holds. In consequence \eqref{25} holds.
\end{proof}

\section{Main results}
We are now in a position to find all continuous solutions $g\colon I\to I$ of equation \eqref{E}. We will do it in three steps.

\begin{thm}\label{thm31}
Assume that $I$ is bounded. If $g\colon I\to I$ is a continuous solution of equation \eqref{E}, then \eqref{23} holds.
\end{thm}

\begin{proof}  
By Lemma \ref{lem11} we see that $g$ is strictly increasing. This jointly with boundedness of $I$ yields that for every $x\in I$ the sequence $(g^n(x))_{n\in\mathbb N}$ converges to a real number. Lemma \ref{lem23} completes the proof.
\end{proof}

\begin{thm}\label{thm32}
Assume that $I$ is a half-line. If $g\colon I\to I$ is a continuous solution of equation \eqref{E}, then there exist $c\in\mathbb R$ such that 
\begin{equation}\label{31}
g(x)=x+c \hspace{3ex}\hbox{ for every }x\in I.
\end{equation}
Moreover, $c\leq 0$ in the case where $\inf I=-\infty$ and $c\geq 0$ in the case where $\sup I=+\infty$.
\end{thm}

\begin{proof} 
From Lemma \ref{lem11} we see that $g$ is strictly increasing. By replacing the function $g$ by the function $\overline{g}\colon -I\to -I$ given by $\overline{g}(x)=-g(-x)$ if necessary, we can assume that $\sup I=+\infty$. 

By Theorem \ref{thm31} we can also assume that $g$ has no fixed point in int$I$. Indeed, if  $g(x_0)=x_0$ for some $x_0\in{\rm int}I$, then $g(x)=x$ for every $x\in (-\infty,x_0]\cap I$, by Theorem \ref{thm31}, and $g$ solves \eqref{E} on $(x_0,+\infty)$.

First, we prove that $g(x)>x$ for every $x\in{\rm int}I$. Suppose the contrary; i.e., $g(x)<x$ for every $x\in{\rm int}I$. Then for every $x\in I$ the sequence $(g^n(x))_{n\in\mathbb N}$ converges to $\inf I$. Since $\inf I$ is a real number, it follows, by Lemma \ref{lem23}, that \eqref{23} holds; a contradiction.

By Lemma \ref{lem24} we see that \eqref{25} holds. From \eqref{en}, Lemma \ref{lem21} and \eqref{24} we obtain
$$
g^{n+3}(x)=b_{n-1}(2g(x)-x)+b_ng(x)-2b_{n-2}x=(n+3)g(x)-(n+2)x
$$
for all $n\in\mathbb N_0$ and $x\in I$. Hence
\begin{equation}\label{32}
\frac{g^{n+3}(y)-g^{n+3}(x)}{y-x}=(n+3)\frac{g(y)-g(x)}{y-x}-(n+2)
\end{equation}
for all $n\in\mathbb N_0$ and $x,y\in I$ with $x\neq y$.

Fix $x\in {\rm int}I$ and choose $y\in I$ and $k\in\mathbb N$ such that $x<y$ and $y<g^k(x)$; it is possible because $\lim_{k\to\infty}g^k(x)=+\infty$. 
Then by the monotonicity of $g$ and \eqref{25} we obtain
$$
0<\frac{g^{n+3}(y)-g^{n+3}(x)}{y-x}\leq\frac{g^{n+k+3}(x)-g^{n+3}(x)}{y-x}=\frac{k(g(x)-x)}{y-x},
$$
and hence $\lim_{n\to\infty}\frac{g^{n+3}(y)-g^{n+3}(x)}{(n+3)(y-x)}=0$. Thus, dividing both sides of \eqref{32} by $(n+3)$ and next tending with $n$ to infinity, we obtain
$$
\frac{g(y)-g(x)}{y-x}=1.
$$
This jointly with continuity of $g$ gives \eqref{31} with $c=g(y)-y>0$.

In conclusion, we have proved that \eqref{31} holds with some $c>0$ in the case where $g$ has no fixed point in int$I$; otherwise \eqref{31} holds with $c=0$.
\end{proof}

\begin{thm}\label{thm33}
Assume that $I=\mathbb R$. If $g\colon I\to I$ is a continuous solution of equation \eqref{E}, then there exists $c\in\mathbb R$ such that either \eqref{31} holds or
\begin{equation}\label{33}
g(x)=-2x+c \hspace{3ex}\hbox{ for every }x\in I.
\end{equation}
\end{thm}

\begin{proof} 
From Lemma \ref{lem12} we see that either $g$ is an increasing bijection from $\mathbb R$ onto $\mathbb R$ or it is a decreasing bijection from $\mathbb R$ onto $\mathbb R$.

First, we consider the case where $g$ is an increasing bijection. 

If there exists $x_0\in\mathbb R$ such that $g(x_0)=x_0$, then both the functions $g|_{(-\infty,x_0]}$ and $g|_{[x_0,+\infty)}$ satisfy equation \eqref{E}. Then applying Theorem \ref{thm32} we conclude that \eqref{31} holds with $c=0$. 

If $g(x)\neq x$ for every $x\in\mathbb R$, then either $g(x)>x$ for every $x\in\mathbb R$ or $g(x)<x$ for every $x\in\mathbb R$. 
Assume that $g(x)>x$ for every $x\in\mathbb R$. Fix $y\in\mathbb R$ and observe that the function $g|_{[y,+\infty)}$ satisfies equation \eqref{E}. By Theorem \ref{thm32} there exists $c>0$ such that $g(x)=x+c$ for every $x\in[y,+\infty)$. Letting with $y$ to $-\infty$ we conclude that \eqref{31} holds. 
In the same way we can prove that \eqref{31} holds with some $c<0$ in the case where $g(x)<x$ for every $x\in\mathbb R$.

Secondly, we consider the case where $g$ is a decreasing bijection. 

By Lemma \ref{lem12} we see that the formula $G=g^{-1}$ defines a strictly decreasing bijection $G\colon\mathbb R\to\mathbb R$. Putting $g^{-3}(x)$ in place of $x$ in \eqref{E} we conclude that
\begin{equation}\label{34}
G^3(x)=\frac{3}{2}G^2(x)-\frac{1}{2}x
\end{equation}
for every $x\in\mathbb R$.

Fix $x\in\mathbb R$ and define a sequence $(x_n)_{n\in\mathbb N_0}$ putting
$$
x_0=x\hspace{3ex}\hbox{and}\hspace{3ex}x_n=G(x_{n-1})\hspace{2ex}\hbox{for every }n\in\mathbb N.
$$
By \eqref{34} we have
\begin{equation}\label{35}
x_{n+3}=\frac{3}{2}x_{n+2}-\frac{1}{2}x_n
\end{equation}
for every $n\in\mathbb N_0$. 
It is clear that we can find unique real constants $A$, $B$ and $C$ (depending on $x$)  such that
\begin{equation}\label{36}
x_n=A\cdot n+B+C\cdot\left(-\frac{1}{2}\right)^n
\end{equation}
for $n\in\{0,1,2\}$. According to \eqref{35} we conclude, by a simple induction, that \eqref{36} holds
for every $n\in\mathbb N_0$. Since $G$ is strictly decreasing, it follows that the sequence $(x_n)_{n\in\mathbb N_0}$ is anti-monotone; i.e., the expression $(-1)^n(x_{n+1}-x_n)$ does not change its sign. This forces $A=0$, and hence
$2G^2(x)-G(x)-x=2x_2-x_1-x_0=2B+\frac{1}{2}C-B+\frac{1}{2}C-B-C=0$. 

In conclusion, we have proved that 
\begin{equation}\label{37}
2G^2(x)-G(x)-x=0
\end{equation}
for every $x\in\mathbb R$. Putting $g^2(x)$ in place of $x$ in \eqref{37} we obtain
$$
2x-g(x)-g^2(x)=0
$$
for every $x\in\mathbb R$. Finally, applying Theorem 9 from \cite{N1974} we conclude that \eqref{33} holds.
\end{proof}

\section{A Zolt\'{a}n Boros\rq{} problem}
In this section we answer the question posed by Zolt\'{a}n Boros of determining all continuous solutions $f\colon (0,+\infty)\to (0,+\infty)$ of equation \eqref{be}. In fact, we determine all continuous solutions $f\colon J\to J$ of equation \eqref{be}, where $J$ is a subinterval of the half-line $(0,+\infty)$; open or closed or closed on one side, possible infinite or degenerated to a single point.

The proof of the next lemma is very easy, so we omit it. 

\begin{lem}\label{lem41}
If $J\subset\mathbb (0,+\infty)$ is an interval and $f\colon J\to J$ is a solution of equation \eqref{be}, then the formula $g=\log\circ f\circ \exp$ defines a function acting from $\log J$ into itself such that \eqref{E} holds for every $x\in \log J$.

Conversely, if $I\subset\mathbb R$ is an interval and $g\colon I\to I$ is a solution of equation \eqref{E}, then the formula $f=\exp\circ g\circ\log$ defines a function acting from $\exp I$ into itself such that \eqref{be} holds for every $x\in \exp I$.
\end{lem}

Lemma \ref{lem41} and Theorems \ref{thm31}--\ref{thm33} give the following answer to the question of Zolt\'{a}n Boros.

\begin{thm}\label{thm42}
Assume $J\subset(0,+\infty)$ is an interval and let  $f\colon J\to J$ be a continuous solution of equation \eqref{be}.
\begin{itemize}
\item[\rm (i)] If $J$ is bounded and $0\not\in{\rm cl }J$, then $f(x)=x$ for every $x\in J$.
\item[\rm (ii)] If $J$ is bounded and $0\in{\rm cl }J$, then there exists $c\in(0,1]$ such that 
\begin{equation}\label{41}
f(x)=cx\hspace{3ex}\hbox{for every }x\in J.
\end{equation}
\item[\rm (iii)] If $J$ is unbounded and $0\not\in{\rm cl }J$, then there exists $c\in[1,+\infty)$ such that \eqref{41} holds.
\item[\rm (iv)] If $J=(0,+\infty)$, then there exists $c\in(0,+\infty)$ such that either \eqref{41} holds or
$$
f(x)=\frac{c}{x^2}\hspace{3ex}\hbox{ for every }x\in (0,+\infty).
$$
\end{itemize}
\end{thm}

\subsection*{Acknowledgment}
This research was supported by University of Silesia Mathematics Department (Iterative Functional Equations and Real Analysis program).

\end{document}